\documentclass[10pt]{amsart}
\usepackage{amsthm,amsfonts,amssymb,amscd,amsbsy,amsmath,graphicx,bbm,tikz}
\usetikzlibrary{decorations.pathreplacing,}

\usepackage[margin=1in]{geometry}

\def\a{\alpha}
\def\b{\beta}

\def\zb{\hat{\textbf{0}}}

\def\d{\delta}

\newcommand{\inv}{^{-1}}

\usepackage{scrextend}

\newtheorem{theorem}{Theorem}
\newtheorem{lemma}{Lemma}
\newtheorem{prop}{Proposition}
\newtheorem{cor}{Corollary}
\newtheorem{rem}{Remark}

\theoremstyle{definition}
\newtheorem{definition}{Definition}

\newcommand{\rk}{\operatorname{rk}}
\begin{document}
\title[Kazhdan--Lusztig polynomials of $\rho$-removed uniform matroids]{A combinatorial formula for Kazhdan--Lusztig polynomials of $\rho$-removed uniform matroids}

\author[K. Lee]{Kyungyong Lee}
\thanks{K.L. was partially supported by the University of Alabama, University of Nebraska--Lincoln, Korea Institute for Advanced Study, and the NSF grant DMS 1800207.}
\address{Department of Mathematics, University of Alabama, Tuscaloosa, AL 35401 U.S.A., and School of Mathematics, Korea Institute for Advanced Study, Seoul 02455 Republic of Korea}
\email{kyungyong.lee@ua.edu; klee1@kias.re.kr}

\author{George D. Nasr}
\address{Department of Mathematics, University of Nebraska--Lincoln, Lincoln, NE 68588, U.S.A.}
\email{george.nasr@huskers.unl.edu}
\author{Jamie Radcliffe}
\thanks{Jamie Radcliffe was supported in part by Simons grant number 429383.}
\address{Department of Mathematics, University of Nebraska--Lincoln, Lincoln, NE 68588, U.S.A.}
\email{jamie.radcliffe@unl.edu}

\begin{abstract}
Let $\rho$ be a non-negative integer. A $\rho$-removed uniform matroid is a matroid obtained from a uniform matroid by removing a collection of $\rho$ disjoint bases.
We present a combinatorial formula for Kazhdan--Lusztig polynomials of $\rho$-removed uniform matroids, using  
skew Young Tableaux. Even for uniform matroids, our formula is new, gives manifestly positive integer coefficients, and is more manageable than known formulas. 
\end{abstract}
\maketitle

\section{Introduction}
The Kazhdan--Lusztig polynomial of a matroid was introduced by Elias, Proudfoot, and Wakefield in 2016 \cite{firstklpolys}, which we define here. Throughout, let $M$ be a matroid, $F$ be a flat of the matroid $M$, $\rk$ be the rank function on $M$, and $\chi_M$ be the characteristic function for a matroid. We denote $M^F$ (respectively $M_F$) for the localization (resepctively contraction) for $M$ at $F$. Then, the Kazhdan-Lusztig polynomial for $M$, denoted $P_M(t)$ is given by the following conditions:
\begin{enumerate}
\item If $\rk M=0$, then $P_M(t)=1$.
\item If $\rk M>0$, then $\deg P_M(t)<{1\over 2} \rk M$. 
\item $\displaystyle t^{rk M}P_M(t\inv)=\sum_{F: \text{ a flat}} \chi_{M^F}(t)P_{M_F}(t)$.
\end{enumerate}

 Since their introduction, these polynomials have drawn active research efforts. Mostly, this is due to their (conjecturally) nice properties, such as positivity and real-rootedness (see \cite{firstklpolys,thag,klum,fanwheelwhirl,equithag}). There has also been much effort put into finding relations between these polynomials or generalizations thereof (see \cite{deletion,kls,pfia}). However, these polynomials have been explicitly calculated only for very special classes of matroids (for instance, see \cite{stirling,klum,klpolys,fanwheelwhirl,qniform}), and yet many of the known formulas have left much room for improvement. In particular, as of now, there is no enlightening interpretation for such coefficients. The goal of this paper is to provide a manifestly positive integral formula for the coefficients of specific matroids by identifying them with a combinatorial object.

In what follows, $U_{m,d}(\rho)$ is the matroid constructed by starting with the uniform matroid of rank $d$ on $m+d$ elements, and removing $\rho$ disjoint bases. Our main result is as follows. 
\begin{theorem}\label{thm:main}
Let ${c^i_{m,d}}(\rho)$ be the $i$-th coefficient for the Kazhdan-Lusztig polynomial for the matroid $U_{m,d}(\rho)$. Then
\[{c^i_{m,d}}(\rho)=\# SkYT_\rho(m+1,i,d-2i+1),\]
where $SkYT_\rho(m+1,i,d-2i+1)$ is the set of legal fillings $\alpha$ on the following skew Young diagram
 \begin{center}
\begin{tikzpicture}[scale=0.2, line width=1pt]
  \draw (-1,0) grid (0,6);
\draw[decoration={brace,raise=7pt},decorate]
 (-1,0) -- node[left=7pt] {$m+1$} (-1,6);
 
  \draw (0,4) grid (4,6);
\draw[decoration={brace,mirror, raise=4pt},decorate]
 (0.1,4) -- node[below=7pt] {$i$} (5,4);
 
  \draw (4,4) grid (5,9);
\draw[decoration={brace,mirror, raise=5pt},decorate]
 (5,4) -- node[right=7pt] {$d-2i+1$} (5,9); 
\end{tikzpicture}
\end{center}
with the additional restrictions that 

\noindent$\bullet$ the top entry of the right-most column of $\alpha$ is 1; or

\noindent$\bullet$ the bottom entry of the right-most column  is greater than $d+\rho$; or

\noindent$\bullet$ the third entry (from the top) of the left-most column  is less than $d+1$.
\end{theorem}

Since $U_{m,d}(\rho)$ is representable, this in turn gives a combinatorial formula for the intersection cohomology Poincar\'{e} polynomial of the corresponding reciprocal plane over a finite field, thanks to \cite{firstklpolys}.

Note that when $\rho=0$, the matroid $U_{m,d}(\rho)$ is in fact $U_{m,d}$, the uniform matroid of rank $d$ on $m+d$ elements. Hence, in this case, Theorem \ref{thm:main} gives a new manifestly positive integral formula\footnote{The first (and only known) manifestly positive integral interpretation for uniform matroids was given in \cite[Remark 3.4]{klpolysequiv}, which requires possibly many Young diagrams.} for the coefficients corresponding to the uniform matroid, which is simply the number of all legal fillings on the diagram above. This is because the bottom entry of the right-most column is always greater than $d$. In this case, we get the following alternative of Theorem 1.\\

\noindent \textbf{Theorem 2.} \textit{Let $c_{m,d}^i$ be the $i$th coefficient of the Kazhdan-Lusztig Polynomial for the uniform matroid of rank $d$ on $m+d$ elements. Then}
$$c_{m,d}^i=\#SkYT(m+1,i,d-2i+1),$$
\textit{where $\#SkYT(m+1,i,d-2i+1)$ is just the number of legal fillings of the above shape, with no other restrictions.}

 When $m+1=2$ or $d-2i+1=2$, this number becomes equal to a well-known number, namely the number of polygon dissections \cite{S}. Hence, when $m+1=d-2i+1=2$, it becomes a Catalan number.\footnote{This connection to polygon dissections was already mentioned in Remark 5.3 of \cite{klpolys}, but with the discovery of our combinatorial object, this fact follows directly from \cite{S}.}

This paper is organized as follows. In section \ref{sec:SYT}, we introduce the central combinatorial object to this paper and prove useful identities about it. In section \ref{sec:pot1}, we prove Theorem \ref{thm:comb_int}, which is a proof of Theorem \ref{thm:main} in the case of $\rho=0$. We then discuss a symmetry condition on the coefficients of the uniform matroid (already mentioned in \cite{klpolysequiv}) in section \ref{sec:symmetries} by using our combinatorial object. Finally, we end the paper with section \ref{sec:rho_removed} by rigorously defining the matroid $U_{m,d}(\rho)$, defining $SkYT_\rho(a,i,b)$, and providing a proof of a Theorem (called Theorem \ref{thm:once_removed}), which is equivalent to Theorem \ref{thm:main}. Note that the statements of the two theorems are in fact equivalent, and also that the theorems imply the $\rho=0$, so between Theorems \ref{thm:comb_int} and \ref{thm:once_removed} we actually give two proofs for the case where $\rho=0$. It is however still worth mentioning Theorem \ref{thm:comb_int} as it can be proved using a lot less work.
 
 \section{Skew Young Tableaux}\label{sec:SYT}
 
 Consider the following shape.
 \begin{center}
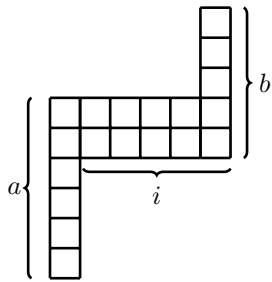
\begin{figure}[h]

\begin{tikzpicture}[scale=0.4, line width=1pt]
  \draw (-1,0) grid (0,6);
\draw[decoration={brace,raise=7pt},decorate]
 (-1,0) -- node[left=7pt] {$a$} (-1,6);
 
  \draw (0,4) grid (4,6);
\draw[decoration={brace,mirror, raise=4pt},decorate]
 (0.1,4) -- node[below=7pt] {$i$} (5,4);
 
  \draw (4,4) grid (5,9);
\draw[decoration={brace,mirror, raise=5pt},decorate]
 (5,4) -- node[right=7pt] {$b$} (5,9); 
\end{tikzpicture}
\caption{The left-most column has height $a$, followed by $i-1$ columns of height 2, followed by the right-most column of height $b$.}
\end{figure}
\end{center} 
A \textit{legal filling }of the above shape involves placing each number from $\{1,2,\dots, a+2i+b-2\}$ into the squares such that the values in the columns and rows strictly increase going down and right, respectively. Note that this is the same restriction on the entries of a standard young tableau, but the above shape does not fit the description of the typical young tableau. We refer to legal fillings of the above shape as \textit{skew young tableau}, and denote $SkYT(a,i,b)$ as the set of such legal fillings. There are some exceptional values we have set conventions for: 
\begin{itemize}
\item If $i=0$, $SkYT(a,i,b)=1$.
\item If $i>0$ and at least one of $a$ or $b$ is less than 2, $SkYT(a,i,b)=0$.
\end{itemize} 

Using this object we will achieve the following result.
\begin{theorem}\label{thm:comb_int}
Let $c_{m,d}^i$ be the $i$th coefficient of the Kazhdan-Lusztig Polynomial for the uniform matroid of rank $d$ on $m+d$ elements. Then
$$c_{m,d}^i=\#SkYT(m+1,i,d-2i+1)$$
\end{theorem}

The proof of this Theorem relies on determining the size of $SkYT(a,i,b)$ explicitly. For this, we take advantage of being able to find the number of standard young tableau explicitly, thanks to the hook length formula. We define $SYT(a,i,k)$ to be the set of standard young tableaux of the following shape.

 
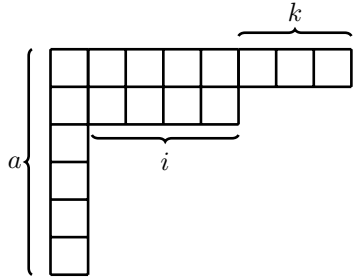
\begin{center}
\begin{figure}[h]

\begin{tikzpicture}[scale=0.5, line width=1pt]
  \draw (-1,0) grid (0,6);
\draw[decoration={brace,raise=7pt},decorate]
(-1,0) -- node[left=7pt] {$a$} (-1,6);
 
  \draw (0,4) grid (4,6);
\draw[decoration={brace,mirror, raise=4pt},decorate]
 (0.1,4) -- node[below=7pt] {$i$} (4,4);

  \draw (4,5) grid (7,6);
\draw[decoration={brace, raise=5pt},decorate]
 (4,6) -- node[above=7pt] {$k$} (7,6); 
\end{tikzpicture}
\end{figure}\label{fig:tablerec}
\end{center} 
\begin{lemma}\label{lem:syt_count}
\[\#SkYT(a,i,b)=\sum_{k=0}^{b-2} (-1)^k{a+2i+b-2 \choose b-k-2} \#SYT(a,i,k), \]

\begin{proof}
 Observe one could build $SkYT(a,i,b)$ by starting with a young diagram $\mu$ with $b-2$ parts of size 1, choosing the elements from $[a+b+2i-2]$ to place in there in increasing order, and then from the remaining numbers, place them in one of $\#SYT(a,i,0)$ ways, giving a tableau $\lambda$, and then attaching $\mu$ to $\lambda$ so that the bottom of $\mu$ is adjacent to the top right of $\lambda$. See the below figure. 
 
 \begin{center}
 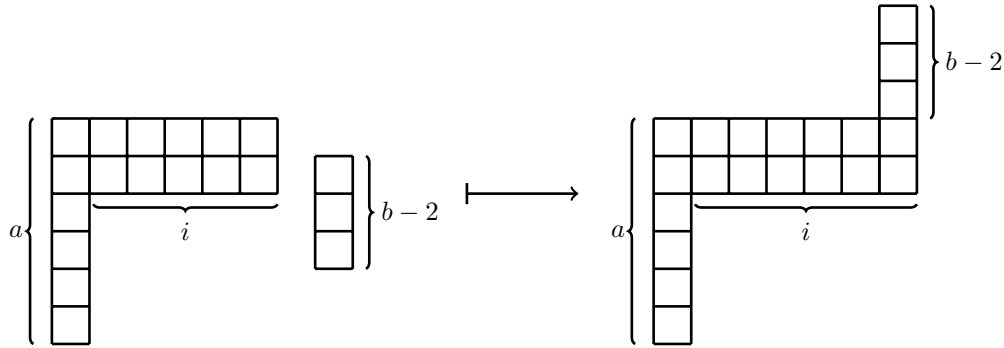
\begin{figure}[h]
\begin{tikzpicture}[scale=0.5, line width=1pt]
  \draw (-1,0) grid (0,6);
\draw[decoration={brace,raise=7pt},decorate]
 (-1,0) -- node[left=7pt] {$a$} (-1,6);
 
  \draw (0,4) grid (5,6);
\draw[decoration={brace,mirror, raise=4pt},decorate]
 (0.1,4) -- node[below=7pt] {$i$} (5,4);
 
  \draw (6,2.) grid (7,5.);
\draw[decoration={brace,mirror, raise=5pt},decorate]
 (7,2) -- node[right=7pt] {$b-2$} (7,5); 
 
  \draw[|->] (10,4) -- (13,4); 
  
  \draw (15,0) grid (16,6);
\draw[decoration={brace,raise=7pt},decorate]
 (15,0) -- node[left=7pt] {$a$} (15,6);
 
  \draw (16,4) grid (21,6);
\draw[decoration={brace,mirror, raise=4pt},decorate]
 (16.1,4) -- node[below=7pt] {$i$} (22,4);
 
  \draw (21,4) grid (22,9);
\draw[decoration={brace,mirror, raise=5pt},decorate]
 (22,6) -- node[right=7pt] {$b-2$} (22,9); 
\end{tikzpicture}
\caption{Tableaux $\lambda$ and $\mu$ combine to give the shape desired skew-symmetric tableau shape.}
\end{figure}
\end{center}
 
 Of course, these pieces are only compatible if the entry in the bottom entry of $\mu$ is smaller than the top right of $\lambda$, so we need to remove the cases not giving legal fillings. By moving the bottom square of $\mu$ the right of the top right piece of $\lambda$, we have a bijection between this case and having a pair of tableau, one standard young tableau with $b-3$ parts of size 1 and the other from $SYT(a,i,1)$, that we wish to remove from the possible count. Of course, this will also remove cases where the second to last entry in $\mu$ is larger than the last entry, which was not accounted for before since we assumed we placed the entries in $\mu$ in increasing order, so we wish to add these cases back in. We can count this in a similar way by counting the number of pairs of standard young tableaux where one is $b-4$ parts of size 1, and then selecting an element from $SYT(a,i,2)$. Continuing this process gives the right hand side of the desired equality in the statement of Lemma \ref{lem:syt_count}, and by an inclusion-exclusion argument, we have also counted the left.

\end{proof}
\end{lemma}

\section{Proof of Theorem \ref{thm:comb_int}}\label{sec:pot1}
Let $a=m+1$ and $b=d-2i+1$. We instead prove the statement with our change of coordinates, that is, we show $c_{a-1,b+2i-1}=\#SkYT(a,i,b)$. Utilizing \cite[Theorem 1.4]{klum}, we already know that
\[c_{a-1,b+2i-1}^i={1\over b+i-1}{b+2i+a-2 \choose i}\sum_{h=0}^{a-2} {b+i+h-1\choose h+i+1}{i-1+h\choose h}.\]  
In Section \ref{sec:symmetries}, we prove Lemma \ref{lem:sym} which gives the identity $\#SkYT(a,i,b)=\#SkYT(b,i,a)$. Hence,
\begin{align*}
\#SYT(b,i,k){a+2i+b-2 \choose a-k-2}&={1\over (b+i-1)}\frac{(b+2i+k)!(k+1)}{(b-2)!i!(b+i+k)(i+k+1)!}{(a+2i+b-2)! \over (a-k-2)!(2i+k+b)!}\\
&={1\over (b+i-1)}\frac{(k+1)(a+2i+b-2)!}{(b-2)!i!(b+i+k)(i+k+1)!(a-k-2)!}\\
&={1\over (b+i-1)}{a+2i+b-2 \choose i}\frac{(k+1)(a+i+b-2)!(b+i+k-1)!}{(b-2)!(b+i+k)!(i+k+1)!(a-k-2)!}\\
&={1\over (b+i-1)}{a+2i+b-2 \choose i}(k+1){a+i+b-2 \choose b+i+k}{b+i+k-1 \choose b-2},\\
\end{align*}
Hence, it suffices to show that 
\[\sum_{k=0}^{a-2} (-1)^k(k+1){a+i+b-2 \choose b+i+k}{b+i+k-1 \choose b-2}=\sum_{h=0}^{a-2} {b+i+h-1\choose h+i+1}{i-1+h\choose h}.\]  
We instead prove that 
\[\sum_{a,b\geq 0}x^ay^b\sum_{k=0}^{a-2} (-1)^k(k+1){a+i+b-2 \choose b+i+k}{b+i+k-1 \choose b-2}=\sum_{a,b\geq 0}x^ay^b\sum_{h=0}^{a-2} {b+i+h-1\choose h+i+1}{i-1+h\choose h}.\]  
We take advantage of the following two things:
\begin{enumerate}
\item $\displaystyle \sum_{k=0}^\infty {k+i\choose k}x^k={1\over (1-x)^{i+1}}$
\item $\displaystyle \sum_{k\geq 0} (-1)^k(k+1)z^k={d\over dz}z\sum_{k\geq 0} (-1)^kz^k={d\over dz}{z\over 1+z}={1\over (1+z)^2}$
\end{enumerate}
Now on the one hand, we have 
\begin{align*}
&\sum_{a,b,k\geq 0}x^ay^b (-1)^k(k+1){a+i+b-2 \choose b+i+k}{b+i+k-1 \choose b-2}\\
&=\sum_{b,k\geq 0}y^b (-1)^k(k+1){b+i+k-1 \choose b-2}\sum_{a\geq 0}x^a{a+i+b-2 \choose b+i+k}\\
&=\sum_{b,k\geq 0}y^b (-1)^k(k+1){b+i+k-1 \choose b-2}x^{k+2}\sum_{a\geq k+2}x^{a-k-2}{a+i+b-2 \choose a-k-2}\\
&={x^2\over (1-x)^{i+1}}\sum_{k\geq 0} (-1)^k(k+1)\left({x\over 1-x}\right)^k\sum_{b\geq 0}{b+i+k-1 \choose i+k+1}\left({y\over 1-x}\right)^b\\
&={x^2\over (1-x)^{i+1}}\sum_{k\geq 0} (-1)^k(k+1)\left({x\over 1-x}\right)^k\left({y\over 1-x}\right)^2\sum_{b\geq 0}{b+i+k+1 \choose i+k+1}\left({y\over 1-x}\right)^b\\
&={x^2y^2\over (1-x)^{i+3}(1-{y\over 1-x})^{i+2}}\sum_{k\geq 0} (-1)^k(k+1)\left({x\over 1-x-y}\right)^k\\
&={x^2y^2\over (1-x)^{i+3}(1-{y\over 1-x})^{i+2}}{1\over (1+{x\over 1-x-y})^2}\\
&={x^2y^2\over (1-x)(1-x-y)^{i}(1-y)^2}.
\end{align*}

On the other hand, we have 
\begin{align*}
&\sum_{a,b}x^ay^b\sum_{h=0}^{a-2} {b+i+h-1\choose h+i+1}{i-1+h\choose h}\\
&=\sum_{a\geq 0}x^a\sum_{h=0}^{a-2} {i-1+h\choose h}\sum_{b\geq 0} y^b{b+i+h-1\choose h+i+1}\\
&=\sum_{a\geq 0}x^a\sum_{h=0}^{a-2} {i-1+h\choose h}y^2\sum_{b\geq 0} y^b{b+i+h+1\choose h+i+1}\\
&={y^2\over (1-y)^{i+2}}\sum_{h\geq 0} {i-1+h\choose h}{1\over (1-y)^h}\sum_{a\geq h+2}x^a\\
&={x^2y^2\over (1-y)^{i+2}(1-x)}\sum_{h\geq 0} {i-1+h\choose h}\left({x\over 1-y}\right)^h\\
&={x^2y^2\over (1-y)^{i+2}(1-x)}{1\over (1-{x\over 1-y})^i}\\
&={x^2y^2\over (1-y)^{2}(1-x)(1-y-x)^i}.
\end{align*} 
$\hfill \square$

\section{Symmetries}\label{sec:symmetries}

This combinatorial realization does more than provide a manifestly positive and integral interpretation for these coefficients. In \cite{klpolysequiv}, they define a new polynomial called the equivariant KL polynomial for the uniform matroid, and use it to observe a surprising symmetry in the coefficients of the equivariant KL polynomial for uniform matroid. If we let $C_{m,d}^i$ be the $i$th coefficient of the equivariant KL polynomial for the uniform matroid of rank $d$ on $m+d$ elements, $C_{m,d}^i=C_{d-2i,m+2i}^i$, remarking that they see ``no philosophical reason why this symmetry should exist'' \cite[Remark 3.5]{klpolysequiv}. They are able to use $C_{m,d}^i$ to recover $c_{m,d}^i$, and so the same symmetry is true for the latter. We recover this  symmetry by observing symmetry in our skew symmetric tableaux.

\begin{lemma}\label{lem:sym}
\[\# SkYT(a,i,b)=\#SkYT(b,i,a)\]
\begin{proof}
Given $\a\in SkYT(a,i,b)$, define $\bar{\a}\in SkYT(b,i,a)$ in the following way. Let $n$ be the maximum value for the entries of the elements of $SkYT(a,i,b)$, and hence also $SkYT(b,i,a)$. Replace each number $i$ in $\a$ with $n+1-i$, and rotate the shape 180 degrees, so that the shape corresponds to the elements of $SkYT(b,i,a)$. This map is necessarily an involution. 

This process is also well defined. Let $x$ and $y$ be two positions in $\a$ containing entries $i,j\in [n]$ respectively. Suppose $x$ and $y$ are positioned so that the entry in $x$ is required to be smaller than the entry in $y$. This is to say that $x$ is to the right or above $y$ (or both). This also gives us that $i<j$. Our above map replaces the entries of $x$ and $y$ with $n+1-i$ and $n+1-j$, and then rotates $\a$ giving us $\bar{\a}$. When we do this, if $x$ was above $y$, it is now below, and likewise with being to the right versus left. Regardless, there relative locations now require the value of $y$ to be less than $x$, which is indeed true since $i<j$, giving this map is indeed well-defined. The figure below gives an example of this map.
\end{proof}\begin{center}
\begin{figure}[h]
 \begin{minipage}{.3\textwidth}
 \begin{center}
 
\begin{tikzpicture}[scale=0.5, line width=1pt]
  \draw (-1,0) grid (0,4);
 \node[] (x1) at (-.5,3.5) {$2$};
 \node[] (x1) at (-.5,2.5) {$3$};
 \node[] (x1) at (-.5,1.5) {$10$};
 \node[] (x1) at (-.5,.5) {$11$};
 
  \draw (0,2) grid (2,4);
 \node[] (x1) at (.5,3.5) {$4$};
 \node[] (x1) at (.5,2.5) {$6$};
 \node[] (x1) at (1.5,3.5) {$5$};
 \node[] (x1) at (1.5,2.5) {$8$};

  \draw (2,2) grid (3,5);
 \node[] (x1) at (2.5,4.5) {$1$};
 \node[] (x1) at (2.5,3.5) {$7$};
 \node[] (x1) at (2.5,2.5) {$9$};
 
\end{tikzpicture}
 \end{center}
 \end{minipage}
 \begin{minipage}{.3\textwidth}
 \begin{center}
 
\begin{tikzpicture}[scale=0.5, line width=1pt]
  \draw (-1,0) grid (0,4);
 \node[] (x1) at (-.5,3.5) {$10$};
 \node[] (x1) at (-.5,2.5) {$9$};
 \node[] (x1) at (-.5,1.5) {$2$};
 \node[] (x1) at (-.5,.5) {$1$};
 
  \draw (0,2) grid (2,4);
 \node[] (x1) at (.5,3.5) {$8$};
 \node[] (x1) at (.5,2.5) {$6$};
 \node[] (x1) at (1.5,3.5) {$7$};
 \node[] (x1) at (1.5,2.5) {$4$};

  \draw (2,2) grid (3,5);
 \node[] (x1) at (2.5,4.5) {$11$};
 \node[] (x1) at (2.5,3.5) {$5$};
 \node[] (x1) at (2.5,2.5) {$3$};
 
\end{tikzpicture}
 \end{center}
 \end{minipage}
 \begin{minipage}{.3\textwidth}
 \begin{center}
 
\begin{tikzpicture}[scale=0.5, line width=1pt]
  \draw (-1,1) grid (0,4);
 \node[] (x1) at (-.5,1.5) {$11$};
 \node[] (x1) at (-.5,2.5) {$5$};
 \node[] (x1) at (-.5,3.5) {$3$};
 
  \draw (0,2) grid (2,4);
 \node[] (x1) at (1.5,2.5) {$8$};
 \node[] (x1) at (1.5,3.5) {$6$};
 \node[] (x1) at (.5,2.5) {$7$};
 \node[] (x1) at (.5,3.5) {$4$};

  \draw (2,2) grid (3,6);
 \node[] (x1) at (2.5,2.5) {$10$};
 \node[] (x1) at (2.5,3.5) {$9$};
 \node[] (x1) at (2.5,4.5) {$2$};
 \node[] (x1) at (2.5,5.5) {$1$};
 
\end{tikzpicture}
 \end{center}
 \end{minipage}
 \caption{The left most tableau is an element of $SkYT(4,3,3)$, the middle tableau replaces each entry $i$ of the left with $11+1-i$, and then rotating gives us the tableau on the right, an element of $SkYT(3,3,4)$.}
 
 \end{figure}
 \end{center}
\end{lemma}

In light of this,we have the following corollary to Theorem \ref{thm:comb_int}. 
\begin{cor}
\[c_{m,d}^i=c_{d-2i,m+2i}^i.\]
\end{cor}

\section{The Kazhdan-Lusztig Polynomials for $\rho$-Removed Uniform Matroids}\label{sec:rho_removed}
In this section, we describe a new matroid in terms of the uniform matroid, and accordingly relate the coefficients of their Kazhdan-Lusztig polynomial, which surprisingly extends our result in Theorem \ref{thm:comb_int}. First, we prove a proposition that justifies the definition of this new matroid.

\begin{prop}\label{prop:remove_disjoint_bases}
For $1<k\leq n$, let $\mathcal{B}={[n]\choose k}$, the subsets of $[n]$ of size $k$. Let $\mathcal{D}\subseteq \mathcal{B}$ be a disjoint collection of sets. Then $\mathcal{B}\setminus \mathcal{D}$ is a basis system for a matroid, that is, it satisfies the axioms for a collection of sets to be a basis for a matroid. 
\begin{proof}

Suppose there exists sets $B,B'\in \mathcal{B}\setminus \mathcal{D}$ that fail the exchange condition for bases. That is, if we let $\{b_1,b_2, \dots, b_\ell\}=B'\setminus B$, there exists a $b\in B\setminus B'$ so that for all $b_i\in B'\setminus B$ we have $B_i:=B-b+b_i\notin \mathcal{B}\setminus \mathcal{D}$. That is, $B_i\in \mathcal{D}$. However, $B-b\subseteq B_i$ for all $i$, and so provided $k>1$, the fact that the $B_i$ are members of a disjoint family of sets implies that $B_i=B_j$ for all $i$ and $j$, which in turns implies $b_i=b_j$. Hence, $B\triangle B'=\{b,b_i\}$, and so $B'=B_i$. That is, $B'\in \mathcal{D}$, a contradiction.
\end{proof}
\end{prop}

\begin{rem}
Proposition \ref{prop:remove_disjoint_bases} proves more than the statement says. Instead of having $\mathcal{D}$ be a disjoint collection, we can instead have it be a collection of sets such that no pair of sets have a symmetric difference of size 2. (The $B_i$ we construct will have this property, since each contains a unique elements $b_i$.) Of course, if $D$ and $D'$ are disjoint, then the size of their symmetric difference is $|D|+|D'|>2$ so long as $k>1$. We keep the proposition as written, as it is more directly applicable for what follows.
\end{rem}

Proposition \ref{prop:remove_disjoint_bases} tells us that if we start with the Uniform Matroid $U_{m,d}$, we can remove any collection of disjoint bases, and the resulting collection of bases still satisfies the axioms to be a basis system for a matroid. Observe that because we are removing disjoint bases, this new matroid we build will still have all sets of size at most $d-1$ being independent. This allows us to define the following matroid.

\begin{definition}
Let $\rho\geq 0$. The \textit{$\rho$-removed Uniform Matroid of rank $d$ on $m+d$ elements} for \[m\geq \begin{cases}1 & \rho=0,1 \\d(\rho-1)& \rho>1\end{cases}\] is the matroid $U_{m,d}(\rho)$ achieved by starting with the uniform matroid of rank $d$ on $m+d$ elements and removing $\rho$ disjoint bases.
\end{definition}

\begin{rem}\label{rem:notation}\leavevmode

\begin{enumerate}

\item We disregard the case for $m=0$ since in this case $SkYT(m+1,i,d-2i+1)$ in empty. For $\rho>1$, the only way to gaurentee we have enough bases to remove is if $m+d\geq d\rho$. 

\item There are some special cases worth discussing. Note that $U_{m,d}(0)=U_{m,d}$, and by convention, $U_{m,0}(\rho)=U_{m,0}$. 

\item Since the bases of $U_{m,d}$ are \textit{all} sets of size $d$, and we are removing a disjoint collection of sets, the choice of basis does not change the resulting matroid. Defining the notation $[d]_\ell:=\{1+\ell,2+\ell,\dots, d+\ell\}$, we take $\{[d]_0,[d]_1,\dots, [d]_{\rho-1}\}$ to be the $\rho$ bases we canonically remove to get $U_{m,d}(\rho)$.

\end{enumerate}
\end{rem}

Before stating the primary theorem for this section, we provide notation for a special subset of $SkYT(a,i,b)$. Let $\overline{SkYT}(a,i,b)$ be the subset of $SkYT(a,i,b)$ such that
\begin{enumerate}
\item 1 appears at the top of the first column.
\item The largest entries are in the bottom $a-2$ positions of the first column.
\end{enumerate}
The second condition in this definition provides a natural bijection between $\overline{SkYT}(a,i,b)$ and $\overline{SkYT}(2,i,b)$ in which we `forget' the bottom $a-2$ entries of the first column in each element of $\overline{SkYT}(a,i,b)$, or `remember' them to go the other way. See the below figure. 
\begin{center}
\begin{figure}[h]
 \begin{minipage}{.2\textwidth}
 \begin{center}
 
\begin{tikzpicture}[scale=0.5, line width=1pt]
  \draw (-1,0) grid (0,4);
 \node[] (x1) at (-.5,3.5) {$2$};
 \node[] (x1) at (-.5,2.5) {$3$};
 \node[] (x1) at (-.5,1.5) {$10$};
 \node[] (x1) at (-.5,.5) {$11$};
 
  \draw (0,2) grid (2,4);
 \node[] (x1) at (.5,3.5) {$4$};
 \node[] (x1) at (.5,2.5) {$6$};
 \node[] (x1) at (1.5,3.5) {$5$};
 \node[] (x1) at (1.5,2.5) {$8$};

  \draw (2,2) grid (3,5);
 \node[] (x1) at (2.5,4.5) {$1$};
 \node[] (x1) at (2.5,3.5) {$7$};
 \node[] (x1) at (2.5,2.5) {$9$};
 
\end{tikzpicture}
 \end{center}
 \end{minipage}
 \begin{minipage}{.2\textwidth}
 \begin{center}
 
\begin{tikzpicture}[scale=0.5, line width=1pt]
  \draw (-1,0) grid (0,4);
 \node[] (x1) at (-.5,3.5) {$1$};
 \node[] (x1) at (-.5,2.5) {$3$};
 \node[] (x1) at (-.5,1.5) {$4$};
 \node[] (x1) at (-.5,.5) {$7$};
 
  \draw (0,2) grid (2,4);
 \node[] (x1) at (.5,3.5) {$2$};
 \node[] (x1) at (.5,2.5) {$8$};
 \node[] (x1) at (1.5,3.5) {$5$};
 \node[] (x1) at (1.5,2.5) {$10$};

  \draw (2,2) grid (3,5);
 \node[] (x1) at (2.5,4.5) {$6$};
 \node[] (x1) at (2.5,3.5) {$9$};
 \node[] (x1) at (2.5,2.5) {$11$};
 
\end{tikzpicture}
 \end{center}
 \end{minipage}
 \begin{minipage}{.2\textwidth}
 \begin{center}
 
\begin{tikzpicture}[scale=0.5, line width=1pt]
  \draw (-1,0) grid (0,4);
 \node[] (x1) at (-.5,3.5) {$1$};
 \node[] (x1) at (-.5,2.5) {$3$};
 \node[] (x1) at (-.5,1.5) {$10$};
 \node[] (x1) at (-.5,.5) {$11$};
 
  \draw (0,2) grid (2,4);
 \node[] (x1) at (.5,3.5) {$4$};
 \node[] (x1) at (.5,2.5) {$6$};
 \node[] (x1) at (1.5,3.5) {$5$};
 \node[] (x1) at (1.5,2.5) {$8$};

  \draw (2,2) grid (3,5);
 \node[] (x1) at (2.5,4.5) {$2$};
 \node[] (x1) at (2.5,3.5) {$7$};
 \node[] (x1) at (2.5,2.5) {$9$};
 
\end{tikzpicture}
 \end{center}
 \end{minipage}
 \begin{minipage}{.2\textwidth}
 \begin{center}
 
\begin{tikzpicture}[scale=0.5, line width=1pt]
  \draw (-1,2) grid (0,4);
 \node[] (x1) at (-.5,3.5) {$1$};
 \node[] (x1) at (-.5,2.5) {$3$};
 
  \draw (0,2) grid (2,4);
 \node[] (x1) at (.5,3.5) {$4$};
 \node[] (x1) at (.5,2.5) {$6$};
 \node[] (x1) at (1.5,3.5) {$5$};
 \node[] (x1) at (1.5,2.5) {$8$};

  \draw (2,2) grid (3,5);
 \node[] (x1) at (2.5,4.5) {$2$};
 \node[] (x1) at (2.5,3.5) {$7$};
 \node[] (x1) at (2.5,2.5) {$9$};
 
\end{tikzpicture}
 \end{center}
 \end{minipage}
 \caption{The two left most tableaux are elements of $SkYT(4,3,3)$, but not $\overline{SkYT}(4,3,3)$, while the third tableau from the left is, whose corresponding member of $\overline{SkYT}(2,3,3)$ is shown on the right.}\label{fig:tableauxinclusion}
 
 \end{figure}
 \end{center}
Because of the irrelevance of the first parameter, we sometimes omit it from the notation writing only $\overline{SkYT}(i,b)$. As a convention, we take $\overline{SkYT}(0,b)=0$.

 The goal for this section is to prove the following Theorem.
\begin{theorem}\label{thm:once_removed}
Let ${c^i_{m,d}}(\rho)$ be the $i$th coefficient for the Kazhdan-Lusztig polynomial for $U_{m,d}(\rho)$. Then 
\[{c^i_{m,d}}(\rho)=\# SkYT(m+1,i,d-2i+1)-\rho \# \overline{SkYT}(i,d-2i+1).\]
\end{theorem}

When $\rho$ equals $0$ or $1$, the positivity of the coefficients is immediate. However, this positivity is always true, and we can in fact identify the coefficients by counting some tableau-like object. We start by proving the following Lemma.
\begin{lemma}\label{lem:inclusion}
Let $m\geq 1$, and $A:= \{0,1,\dots, m-1\}$. Then there exists an inclusion \[\iota: A\times SkYT(2,i,b)\hookrightarrow SkYT(m+1,i,b).\]
\end{lemma}
\begin{proof}
Let $\a\in SkYT(2,i,b)$. There is a natural way of viewing $\a$ as an element of $SkYT(m+1,i,b)$ in a way similar to what is described in Figure \ref{fig:tableauxinclusion}---attach to $\a$ a column of $m-1$ squares, placing in them the largest possible numbers (in increasing order) of the entries appearing in an element of $SkYT(m+1,i,b)$. For future reference, we refer to the $1\times (m-1)$ column as $\mu$, and refer to this described image of $\a$ as $\bar{\a}$. Note the following facts:
\begin{itemize}
\item The entry in the bottom right corner of $\a$ is the largest number in the tableau. This number is $n:=2+2(i-1)+b=2i+b$.
\item $n$ is smaller than every entry in $\mu$, and every entry in $\mu$ is larger than every element in $\a$. The elements of $\mu$ are $\{n+1,n+2,\dots, n+m-1\}$.
\end{itemize}

We define an action on the locations of the numbers in $\bar{\a}$ by the elements of $A$, which we denote $i\cdot \bar{\a}$ for $i\in A$. The element $i\cdot \bar{\a}\in SkYT(m+1,i,b)$ is defined by starting with $\bar{\a}$, removing $n+i$ from $\mu$ and placing it where $n$ is, shifting all entries of $\mu$ down, and then placing $n$ at the top of $\mu$. The action is well-defined by the itemized facts above.

Hence, we may define the map $\iota: (i,\a)\mapsto i\cdot \bar{\a}$. To see why this map is an inclusion, simply note that any two distinct $\a,\b\in SkYT(2,i,b)$ \textit{must} disagree in a location other than the bottom right corner, as both are required to have $n$ there. This position will never change value by $\iota$. Then it is immediate that $\iota$ sends $(i,\a)$ and $(j,\b)$ to different elements since the outputs of both will still disagree in the position that $\a$ and $\b$ did. 
\end{proof}

This proves much more than we need. However, the tools used in the proof will be beneficial in showing the given subtraction in Theorem \ref{thm:once_removed} is positive. First, define a distinguished subset of $SkYT(m+1,i,d-2i+1)$, which we denote $SkYT_\rho(m+1,i,d-2i+1)$. Every $\a\in SkYT_\rho(m+1,i,d-2i+1)$ must satisfy at least one of the following conditions.
\begin{itemize}
\item the top entry of the right-most column of $\alpha$ is 1; or

\item the bottom entry of the right-most column  is greater than $d+\rho$; or

\item the third entry (from the top) of the left-most column  is less than $d+1$.
\end{itemize}
We then have the following proposition.

\begin{prop}
\[\#SkYT_\rho(m+1,i,d-2i+1)=\#SkYT(m+1,i,d-2i+1)-\rho\#\overline{SkYT}(m+1,i,d-2i+1).\]
\end{prop}

\begin{proof}
Let $\a\in \overline{SkYT}(m+1,i,d-2i+1)$. Hence, in particular, 1 is at the top of the left column and the largest possible elements are in the left tail (by tail, we mean the the entries starting at the third entry from the top). Let $i\in \{0,1,\dots, \rho-1\}$. Utilizing the notation from Lemma \ref{lem:inclusion}, we have that $i\cdot \a$ has a 1 in the top left position, has $d+1+i$ at the bottom of the right column, and the elements of $\{d+1,d+2,\dots, m+d\}\setminus \{d+1+i\}$ in the left tail. 

Let $\mathcal{S}=\{i\cdot \a: \a\in \overline{SkYT}(m+1,i,d-2i+1), i \in \{0,1,\dots, \rho-1\}\}$. Our work from Lemma \ref{lem:inclusion} gives \[\#\mathcal{S}=\rho\#\overline{SkYT}(m+1,i,d-2i+1).\]
Hence, $\#SkYT(m+1,i,d-2i+1)-\rho\#\overline{SkYT}(m+1,i,d-2i+1)$ counts the number of elements in $SkYT(m+1,i,d-2i+1)\setminus\mathcal{S}$. Such elements are exactly described by the elements in $SkYT_\rho(m+1,i,d-2i+1)$.
\end{proof}

It is now equivalent to state Theorem \ref{thm:once_removed} as our primary result.\\

\noindent \textbf{Theorem 1.} \textit{
Let ${c^i_{m,d}}(\rho)$ be the $i$th coefficient for the Kazhdan-Lusztig polynomial for the matroid $U_{m,d}(\rho)$. Then
\[{c^i_{m,d}}(\rho)=\# SkYT_\rho(m+1,i,d-2i+1).\]}
Summarized, we really have the following corollary to Theorem \ref{thm:once_removed}.
\begin{cor}
\[0\leq {c^i_{m,d}}(\rho)\leq c^i_{m,d}(\rho-1)\leq \dots\leq c^i_{m,d}(1)\leq c^i_{m,d}.\]
\end{cor}
This corollary suggests that having less bases may lead to weakly smaller coefficients.

The proof of Theorem \ref{thm:once_removed} will be proved by using the definition of Kazhdan-Lusztig polynomials directly. This means having an understanding of the flats, localizations, contractions, and the characteristic polynomials for the localizations and contractions of $U_{m,d}(\rho)$ is necessary. We provide these first, along with other important identities, in the subsections that follow. The final subsection will provide the proof for Theorem \ref{thm:once_removed}, and hence prove Theorem \ref{thm:main}.

\subsection{Flats, Contractions, Localizations, and Characteristic Polynomials}\leavevmode

Throughout, let $F$ be a flat. For a matroid $M$, recall that $M^F$ (respectively, $M_F$) denotes the localization (respectively, contraction) of $M$ at $F$. By $M^F$, we mean the matroid with ground set $F$, whose independent sets are those subsets of $F$ that are also independent in $M$. By $M_F$, we mean the matroid with ground set $M\setminus F$, whose independent sets are those subsets whose union with a basis for $F$ is independent in $M$. 

When $M=U_{m,d}$, the localizations and contractions are well understood:
\[(U_{m,d})^F=\begin{cases} U_{m,d} & F=[m+d]\\ U_{0,|F|} & F\neq [m+d] \end{cases},\]
and 
\[(U_{m,d})_F=\begin{cases} U_{0,0} & F=[m+d]\\ U_{m,d-|F|} & F\neq [m+d] \end{cases}.\]

To see this, recall that the flats of $U_{m,d}$ consists of the set $[m+d]$ along with all subsets of $[m+d]$ of size at most $d-1$. Also, note that localizations treat $F$ as the ground set and contractions treat $F$ as the rank 0 element (sometimes referred to as $\hat{\textbf{0}}$) since, in this case, every flat is also independent (that is, a basis for itself).

The corresponding equations for $U_{m,d}(\rho)$ can also be described in a similar manner, though require a bit more work to see. The flats for this matroid are almost the same as the flats for the uniform matroid. Utilizing our notation from Remark \ref{rem:notation}, recall that $[d]_\ell$ is not longer independent. This means all of its subsets of size $d-1$ are no longer flat. We also get that $[d]_\ell$ is now flat of rank $d-1$. (Though $[d]_\ell$ is no longer independent, all subsets of it is.)

\begin{prop}\label{prop:rest_local_orum}
\[(U_{m,d}(\rho))^F=\begin{cases} U_{m,d}(\rho) & F=[m+d]\\ U_{1,d-1} & F=[d]_\ell, \text{ for some $\ell$} \\ U_{0,|F|} & \textit{otherwise}  \end{cases}\]
and 
\[(U_{m,d}(\rho))_F=\begin{cases} U_{m,d}(\rho) & F=\emptyset\\ U_{m,d-|F|}(1) & \emptyset\nsubseteq F\subsetneq [d]_\ell, \text{ for some $\ell$}\\ U_{m-1,1} & F=[d]_\ell, \text{ for some $\ell$} \\ (U_{m,d})_F & \textit{otherwise.}  \end{cases}\]

\end{prop}
Note that when $\rho=0$ these formulas still apply to $U_{m,d}(0)=U_{m,d}$, as in this case there are no $[d]_\ell$ removed from $U_{m,d}$.
\begin{proof}
For the localization, the only new case necessary to mention in comparison to the uniform case is for $F=[d]_\ell$; the other cases follow from the uniform case. By the disjointness of the bases we remove, it suffices to prove the case for $F=[d]_0=[d]$. The localization of this matroid at $[d]$ treats $[d]$ as the ground set, with independent sets being those that are independent in $U_{m,d}(\rho)$. We know every \textit{proper} subset of it is independent by definition of $U_{m,d}(\rho)$, giving $ U_{1,d-1}$.

Now for the contraction. The removal of $[d]_\ell$ does not effect which sets are independent in $(U_{m,d}(\rho))_F$ when $F\nsubseteq [d]_\ell$ for any $\ell$. For the case where $F=[d]$, we want the subsets of $S:=\{d+1,d+2,\dots, d+m\}$ such that their union with a basis for $[d]$ is independent in $U_{m,d}(\rho)$. The bases for $[d]$ are the elements of ${[d]\choose d-1}$, and since bases in $U_{m,d}(\rho)$ not disjoint with $[d]$ are in $U_{m,d}(\rho)$, this means the desired subsets of $S$ are the empty set and every singleton of $S$. This gives a matroid isomorphic to $U_{m-1,1}$. We get the same result for $F=[d]_\ell$ for any $\ell$. Finally, when $F\subsetneq [d]$, note that $F$ is independent, and hence a basis for itself. Thus, the independent sets for $(U_{m,d}(\rho))_F$ are the subsets $X$ of $T:=[m+d]\setminus F$ so that $X\cup F$ is independent in $U_{m,d}(\rho)$. That is, $|X|\leq d-|F|$. When $|X|<d-|F|$, $|X\cup F|<d$ and every subset of $[m+d]$ of size smaller than $d$ is independent. When $|X|=d-|F|$, $X\cup F$ is a basis for $U_{m,d}(\rho)$ if and only if $X\cup F\neq [d]$. Note that if $\ell\geq 1$, $F\cap [d]_\ell\subseteq [d]\cap [d]_\ell=\emptyset$. That is, we get a matroid isomorphic to $U_{m,d-|F|}(1)$. An equivalent argument works for when $F\subsetneq [d]_\ell$ for any $\ell$.
\end{proof}

With these in mind, we can now compute the characteristic equation for all localizations and restructions for $U_{m,d}(\rho)$. However, by Proposition \ref{prop:rest_local_orum}, we equivalently just need to find the characteristic polynomial for $U_{m,d}$ and $U_{m,d}(\rho)$. 

First, recall that for a matroid $M$, the characteristic polynomial is given by 
\[\chi_M(t)=\sum_{F\in L(M)} \mu_{L(M)}(\hat{\textbf{0}},F)t^{\rk M -\rk F},\]

where $L(M)$ is the lattice of flats for matroid $M$.

The case when $M$ is uniform is well understood as $\mu_M$ is understood on the boolean lattice---if $\mathcal{B}^n$ is the boolean lattice on $n$ elements $\mu_{\mathcal{B}^n}(\hat{\textbf{0}},\hat{\textbf{1}})=(-1)^n$. This is useful because for uniform matroids, the interval $[\emptyset,F]$ is a boolean lattice for every flat $F$ except $[m+d]$, and when $F=[m+d]$, we know $\mu_{L(M)}(\hat{\textbf{0}},[m+d])$ is determined by the fact that $\chi_M(1)=0$. This gives
\[\chi_{U_{m,d}}(t)=(-1)^d{m+d-1\choose d-1}+\sum_{i=0}^{d-1} (-1)^i{m+d\choose i}t^{d -i}.\]

We can use the same information to find $\chi_{U_{m,d}(\rho)}$, keeping track of flats related to $[d]_\ell$, which are not independent as they are in the uniform case. 
\begin{prop}\label{prop:charpoly}
\[\chi_{U_{m,d}(\rho)}(t)=(-1)^d{m+d-1\choose d-1}-(-1)^d\rho+t(-1)^{d-1}\left({m+d\choose d-1}-\rho\right)+\sum_{i=0}^{d-2} (-1)^i{m+d\choose i}t^{d -i}.\]
\end{prop}
Again, note that the formula works even in the case of $\rho=0$.
\begin{proof}
For convenience, we omit subscripts for $\chi$ and $\mu$, since throughout we work in $U_{m,d}(\rho)$.
The terms of degree at least 2 follows from the uniform case since in $U_{m,d}(\rho)$, every set of size at most $d-2$ is still flat, since every flat of size $d-1$ is independent. The term of degree one comes from summing $\mu(\zb, F)$ for flats $F$ of rank $d-1$. Recall that these flats are given by $\{[d]_0,[d]_1,\dots, [d]_\ell\}$ and all elements of ${[m+d]\choose d-1}$ not contained in any $[d]_\ell$. When $F$ is one of the latter described flats, it follows from the uniform case that $\mu(\zb, F)=(-1)^{d-1}$. Otherwise, 
\begin{align*}
\mu(\zb,[d])&=-\sum_{\zb\leq F<[d]}\mu(\zb,F)\\
&=-\sum_{i=0}^{d-2}(-1)^i{m+d\choose i}\\
&=(-1)^d+d(-1)^{d-1}.
\end{align*}
Thus the coefficient linear term for $\chi$ is given by 
\begin{align*}
\rho(-1)^d+\rho d(-1)^{d-1}+(-1)^{d-1}\left({m+d\choose d-1}-\rho{d\choose d-1}\right)=(-1)^{d-1}{m+d\choose d-1}-\rho(-1)^{d-1}.
\end{align*}

For the constant term, it is equivalent to negate the sum over $\mu(\zb, F)$ for all flats $F\neq [m+d]$. This gives 
\begin{align*}
-\sum_{i=0}^{d-2} (-1)^i{m+d\choose i}-(-1)^{d-1}{m+d\choose d-1}-\rho(-1)^d&=-\sum_{i=0}^{d-1} (-1)^i{m+d\choose i}-\rho(-1)^d\\
&=(-1)^d{m+d-1\choose d-1}-\rho(-1)^d.
\end{align*}
\end{proof}

It will be helpful to restate this proposition in the following way for when we prove Theorem \ref{thm:once_removed}.
\begin{prop}\label{prop:charpoly_restated} (Proposition \ref{prop:charpoly} restated.)
\[ [t^i]\chi_{U_{m,d}(\rho)} = \begin{cases} (-1)^d{m+d-1\choose d-1}-\rho(-1)^d & i=0\\
                                           (-1)^{d-1}{m+d\choose d-1}-\rho(-1)^{d-1} & i=1\\
                                           (-1)^{d-i}{m+d\choose d-i} & 2\leq i\leq d \end{cases}\]
\end{prop}

\subsection{Useful Identities}\leavevmode

In this section, we provide two identities---one involving $\#SkYT$, the other involving $\#\overline{SkYT}$---whose proofs will be similar. We first discuss the identity pertaining to $\#SkYT$. 
\begin{lemma}\label{lem:SkYT_Identity}
If $i\geq 1$, then
\[0=(-1)^{d-i}{m+d\choose d-i}+\sum_{j=0}^{d-1}\sum_{k=0}^i(-1)^{j-i+k}{j\choose j-i+k}{m+d \choose j}\#SkYT(m+1,k,d-j-2k+1).\]
\end{lemma}

In proving this Lemma, it will be first useful to have the following result, which is in a sense the dual to Lemma \ref{lem:syt_count}.
\begin{lemma}\label{lem:skyt_to_sky_count}
\[\# SYT(m+1,k,d-2k-p-1)=\sum_{j=0}^{d-2k-1}(-1)^{d-1+j}{m+d-p\choose j-p}\# SkYT(m+1,k,d-j-2k+1).\]
\end{lemma}
\begin{proof}
The proof is a similar inclusion-exclusion proof to what was provided in Lemma \ref{lem:syt_count}. Starting with the term for $j=d-2k-1$, consider choosing a pair of tableau. The first tableau is a row with $d-2k-p-1$ squares, with entries selected from $[m+d]$. We call this tableau $\mu$. To get the second tableau, which we call $\lambda$ choose an element of $SkYT(m+1,k,2)$, using the numbers in $[m+d]$ not in the entries of $\mu$.

Our goal now is to attach the left block of $\mu$ (whose entry is denoted $i$) to the right of the top right block of $\lambda$ (whose entry is denoted $j$) in order to build an element of $SYT(m+1,k,d-2k-p-1)$. This only works, of course, if $j<i$. The cases where $j>i$ are in bijection with picking a pair of tableau similar to the ones selected before, but now with a row with $d-2k-p-2$ entries and an element of $SkYT(m+1,k,3)$. But here, there will be a scenario where the top right of the element of $SkYT(m+1,k,4)$ will be smaller than the left of the row with $d-2k-p-2$ entries, but these are counted with a pair of an element from $SkYT(m+1,k,3)$ and a row with $d-2k-p-3$ entries. Continuing this alternating sum gives the desired result.
\end{proof}

We also will find the following integral useful to know:
\begin{prop}\label{prop:integral_identity}
For positive integers $a$ and $b$,
\[\int_0^{-1} x^a(1+x)^b\ dx={(-1)^{a+1}b!\over (a+1)(a+2)\cdots (a+b+1)}.\]
\end{prop}
\begin{proof}
Apply integration by parts by differentiating $(1+x)^b$ and antidifferentiating $x^a$. Ignoring limits of integration, this yields 
\[{(1+x)^bx^{a+1}\over a+1}-{b\over a+1}\int x^{a+1}(1+x)^{b-1}\ dx\]
Apply the same integration by parts again to the integral successively, until you get the final integral 
\[(-1)^b{b!\over (a+1)(a+2)\cdots(a+b)}\int x^{a+b}\ dx=(-1)^b{b!\over (a+1)(a+2)\cdots(a+b+1)}x^{a+b+1}\]

Every other term that appears before this has either a factor of $x$ or $(1+x)$, the all these terms are eliminated when we incorporate the limits of integration, while the last term becomes our desired result.
\end{proof}
\begin{proof}[Proof of Lemma \ref{lem:SkYT_Identity}]

First, we note that \[{j\choose j-i+k}{m+d \choose j}={m+d-i+k\choose j-i+k}{m+d\choose i-k}.\]

This allows us to push the summand indexed by $j$ past one of the binomial coefficients giving an inner sum of
\[\sum_{j=0}^{d-1}(-1)^j{m+d-i+k \choose j-i+k}\#SkYT(m+1,k,d-j-2k+1).\]
Note that if $k>0$, terms for $j>d-2k-1$ are zero as the last input to $SkYT$ will be less than 2, so taking $p=i-k$ in Lemma \ref{lem:skyt_to_sky_count} gives that this sum over $j$ is equal to $(-1)^{d-1}\# SYT(m+1,k,d-k-i-1)$.  So we have now reduced this double sum in the statement of Lemma \ref{lem:SkYT_Identity} to 
\[(-1)^{d-i-1}\sum_{k=0}^i (-1)^{k}{m+d\choose i-k}\# SYT(m+1,k,d-k-i-1).\]

Utilizing the hook-length formula, one can verify that 
\[{m+d\choose i-k}\# SYT(m+1,k,d-k-i-1)={(m+d)!\over(m+d-i)(m-1)!(d-i)!i!}{d-i-k\over m+k}{i\choose k}.\]

Hence, the desired result follows when we show 
\begin{align}
{m+d\choose d-i}{(m+d-i)(m-1)!(d-i)!i!\over (m+d)!}=\sum_{k=0}^i(-1)^k{i\choose k} {d-i-k\over m+k}.\label{eq:altsum_bin}
\end{align}

Consider the function
\[f(x,y)=\sum_{k=0}^i{i\choose k} x^{m+k-1}y^{d-i-k}.\]

On the one hand, taking a derivative with respect to $y$, evaluating at $y=1$, then anti-differentiating with respect to $x$ with lower integral limit 0 and upper integral limit $-1$, we recover the ride side of equation \eqref{eq:altsum_bin} with $f(x,y)$. On the other hand, we can find a closed form for $f(x,y)$ first:
\begin{align*}
f(x,y)&=\sum_{k=0}^i(-1)^k{i\choose k} x^{m+k-1}y^{d-i-k}\\
&=x^{m-1}y^{d-i}\sum_{k=0}^i{i\choose k} \left({x\over y}\right)^{k}\\
&=x^{m-1}y^{d-i}\left(1+{x\over y}\right)^i\\
\end{align*}

Using this explicit version of $f(x,y)$, define $g(x)$ by \[g(x):={d\over dy}f(x,y)\vert_{y=1}=(d-i)x^{m-1}(1+x)^i-x^mi(1+x)^{i-1}.\]

Hence, $\displaystyle\int_0^{-1}g(x)\ dx$ should yield the desired result a closed form for the right side of equation \eqref{eq:altsum_bin}. So long as $i\geq 1$, we can apply Proposition \ref{prop:integral_identity} to get
\[{(d-i)(-1)^mi! \over m(m+1)\cdots(m+i)}-{(-1)^{m+1}i! \over (m+1)(m+2)\cdots(m+i)}={(-1)^mi!(m-1)!(d-i+m)\over (m+i)!}.\]
One can verify this is the left side of equation \eqref{eq:altsum_bin}.
\end{proof}

We have an essentially equivalent identity for $\overline{SkYT}$.

\begin{lemma}\label{lem:SkYTbar_Identity} Suppose $i\geq 1$. Then
\begin{align*}
(-1)^{d-i-1}i{d\choose i+1}+\sum_{j=0}^{d-1}\sum_{k=1}^i(-1)^{j-i+k}{j\choose j-i+k}{d \choose j}\#\overline{SkYT}(k,d-j-2k+1)\\
=\begin{cases}0 & i>1\\(-1)^{d-2} & i=1\end{cases}.
\end{align*}
\end{lemma}

The proof for this is very much similar to Lemma \ref{lem:SkYT_Identity}, but especially due to the dependence on the value of $i$, it is worth at least outlining aspects of the proof. 

Here is the corresponding version of Lemma \ref{lem:skyt_to_sky_count}, which has the same proof of Lemma \ref{lem:skyt_to_sky_count}.

\begin{lemma}\label{lem:barskyt_to_sky_count}
\[\# SYT(2,k,d-2k-p-1)=\sum_{j=0}^{d-2k-1}(-1)^{d-1+j}{m+d-p\choose j-p}\# \overline{SkYT}(k,d-j-2k+1).\]
\end{lemma}

\begin{proof}[Proof of Lemma \ref{lem:SkYTbar_Identity}]
The start of this proof works similarly to the proof of Lemma \ref{lem:SkYT_Identity}. First, note that \[{j\choose j-i+k}{d \choose j}={d-i+k\choose j-i+k}{d\choose i-k},\]
and so we can push the summand over $j$ past a binomial coefficient getting an inner sum of 
\[\sum_{j=0}^{d-1}(-1)^j{d-i+k \choose j-i+k}\#\overline{SkYT}(k,d-j-2k+1),\]
which equals $(-1)^{d-1}\#SYT(2,k,d-k-i-1)$ by Lemma \ref{lem:barskyt_to_sky_count}. Hence our double sum from the statement of Lemma \ref{lem:SkYT_Identity} becomes 
\[(-1)^{d-i-1}\sum_{k=1}^i (-1)^{k}{d\choose i-k}\# SYT(2,k,d-k-i-1).\]
It is worth noting this can not be recovered from the proof of Lemma \ref{lem:SkYT_Identity}, and so at this point, the proof of this Lemma diverges slightly though we employ similar strategies.

We again apply the hook-length formula to get
\[{d\choose i-k}\# SYT(2,k,d-k-i-1)={d!\over i!(d-i+1)!}{(d-i+k+1)(d-k-i)\over k+1}{i\choose k}\]
and hence proving the Lemma for the case where $i>1$ is equivalent to showing
\begin{align}\label{eq:altsum_bin_bar}
i{d\choose i+1}{i!(d-i+1)!\over d!}=\sum_{k=1}^i (-1)^{k+1}{(d-i+k+1)(d-k-i)\over k+1}{i\choose k}.
\end{align}

This time, we define a function $f(x,y,z)$ so that \[f(x,y,z)=\sum_{k=0}^i{i\choose k} x^{d-i+k+1}y^{d-k-i}z^k.\]
When we differentiate $f$ and $x$ and $y$ and evaluate both at 1, then integrate with respect to $z$ from 0 to $-1$, and remove the term corresponding to $k=0$, we recover the right of equation \eqref{eq:altsum_bin_bar}. On the other had, we can find $f$ explicitly:
\begin{align*}
f(x,y,z)&=\sum_{k=0}^i{i\choose k} x^{d-i+k+1}y^{d-k-i}z^k\\
&=x^{d-i+1}y^{d-i}\sum_{k=0}^i{i\choose k} \left({xz\over y}\right)^k\\
&=x^{d-i+1}y^{d-i}\left(1+{xz\over y}\right)^i.\\
\end{align*}
Define $g(z)$ so that 
\[g(z):={d\over dx}\left({d\over dy} f(x,y,z)\vert_{y=1}\right)\vert_{x=1}=(d-i+1)(d-i)(1+z)^i-2i(1+z)^{i-1}z-i(i-1)(1+z)^{i-2}z^2.\]
(Note here that it is important we are in the case where $i>1$ due to the exponent on the last term of $g(z)$.)
We can directly apply Proposition \ref{prop:integral_identity} to get that
\begin{align*}
\int_0^{-1}g(z)\ dz &= -(d-i+1)(d-i){i!\over (i+1)!}-2i{(i-1)!\over (i+1)!}+i(i-1){2(i-2)!\over (i+1)!}\\
&=-{i!(d-i+1)(d-i)\over (i+1)!}.
\end{align*}
This will not be the right side of equation \eqref{eq:altsum_bin_bar}, because we have to remove the $k=0$ term appearing in the sum in equation \eqref{eq:altsum_bin_bar} first. This term is $-(d-i+1)(d-i)$, so we get \[{i!(d-i+1)(d-i)i\over (i+1)!},\] which one can verify is the left of equation \eqref{eq:altsum_bin_bar}.

The case for $i=1$ can be simplified from the $i>1$ case since 
\[(-1)^{d-i-1}\sum_{k=1}^i (-1)^{k}{d\choose i-k}\# SYT(2,k,d-k-i-1)=
(-1)^{d-1}\# SYT(2,1,d-3).\]
Hence, when $i=1$, the left side of the equality in the statement of Lemma \ref{lem:SkYTbar_Identity} becomes
\[(-1)^{d-2}{d\choose 2}+(-1)^{d-1}\# SYT(2,1,d-3)=(-1)^{d-1}\left({(d+1)(d-2) -d(d-1)\over 2}\right)=(-1)^{d-2}.\]
\end{proof}

\subsection{Proof of Theorem \ref{thm:once_removed}}
\begin{proof}
Let $M=U_{m,d}(\rho)$. Recall that the definition for the KL polynomial is that it satisfies the following recurrence,
\[\displaystyle t^{rk M}P_M(t\inv)=\sum_{F \text{ a flat}} \chi_{M^F}(t)P_{M_F}(t),\]
which may be rewritten as 
\[\displaystyle t^{rk M}P_M(t\inv)-P_M(t)=\sum_{F \text{ a non-empty flat}} \chi_{M^F}(t)P_{M_F}(t).\]

Recall that $\deg P(t)<{1\over 2} d$, and so the power of each monomial in $t^{d}P_M(t\inv)$ is strictly larger than ${1\over 2}d$. Hence, our goal is to show that for $0\leq i< {1\over 2} d$ we have 
\begin{align}\label{eq:step1}
-\#SkYT(m+1,i,d-2i+1)+\rho\#\overline{SkYT}(i,d-2i+1)=[t^i]\sum_{F \text{ a non-empty flat}} \chi_{M^F}(t)P_{M_F}(t).
\end{align}

Using our work from Proposition \ref{prop:rest_local_orum}, and consolidating common factors involving the various flats $[d]_\ell$, we can rewrite the right of equation \eqref{eq:step1} to be 
\begin{align}\label{eq:step2}
[t^i]\chi_{U_{m,d}(\rho)}+\rho[t^i]\chi_{U_{1,d-1}}P_{U_{m-1,1}}+\rho\sum_{\substack{\emptyset<F<[d]}}[t^i]\chi_{U_{0,|F|}}P_{U_{m,d-|F|}(1)}+\sum_{\substack{\emptyset<F<[m+d]\\ F\nsubseteq [d]_\ell\ \forall \ell}}[t^i]\chi_{U_{0,|F|}}P_{{U_{m,d-|F|}}},
\end{align}
where the first term corresponds to the case where $F=[m+d]$, and the second where $F=[d]_\ell$ for any $\ell$. Here, we are leveraging the fact that the members of the set $\{[d]_0,[d]_1,\dots, [d]_{\rho-1}\}$ are disjoint, giving rise to the multiples of $\rho$ appearing in \eqref{eq:step2}.

By Proposition \ref{prop:charpoly_restated}, we are required to break this up into three case: $i=0$, $i=1$, and $2\leq i<d/2$ if we are to write this out explicitly. Note we can write everything explicitly except $P_{{U_{m,d-|F|}(1)}}$. Hence, we proceed by induction on the matroid rank $d$, noting that $d>d-|F|$ since for the corresponding summand $\emptyset\subsetneq F\subsetneq[d]$.

We now rewrite equation \eqref{eq:step2} for the three cases. Using the Kronecker Delta function $\delta(i,j)=\begin{cases} 1& i=j\\ 0& i\neq j\\\end{cases}$ we can combine the cases for $i=1$ and $2\leq i<d/2$.
\begin{enumerate}
\item[$i=0$:] 
\begin{align*}
&(-1)^d{m+d-1\choose d-1}-\rho(-1)^{d}+\rho(-1)^{d-1}{d-1\choose d-2}\\
&+\rho\sum_{j=1}^{d-2}{d\choose j}(-1)^j(\#SkYT(m+1,0,d-j+1)-\#\overline{SkYT}(0,d-j+1))\\
&+\sum_{j=1}^{d-1}\left({m+d\choose j}-\rho{d\choose j}\right)(-1)^j\#SkYT(m+1,0,d-j+1)
\end{align*}

\item[$i>0$:]\begin{align*}
&(-1)^{d-i}{m+d-1\choose d-i}-\rho(-1)^{d-1}\delta(i,1)+\rho(-1)^{d-1-i}{d\choose d-1-i}\\
&+\rho\sum_{j=1}^{d-2}{d\choose j}\sum_{k=0}^i(-1)^{j-i+k}{j\choose j-i+k}(\#SkYT(m+1,k,d-j-2k+1)-\#\overline{SkYT}(k,d-j-2k+1))\\
&+\sum_{j=1}^{d-1}\left({m+d\choose j}-\rho{d\choose j}\right)\sum_{k=0}^i(-1)^{j-i+k}{j\choose j-i+k}\#SkYT(m+1,k,d-j-2k+1)
\end{align*}
\end{enumerate}

In both cases, the sum running from $j=1$ to $j=d-2$ is the summand in equation \eqref{eq:step2} over $\emptyset<F<[d]$, since the flats contained in $[d]$ have size at most $d-2$, as $[d]$ is not independent, making every element of \({[d]\choose d-1}\) not a flat. The latter sum running from $j=1$ to $j=d-1$ corresponds to the summand in equation \eqref{eq:step2} over $\emptyset<F<[m+d]$ such that $F\subsetneq [d]$. This justifies the binomial coefficients for both cases.

For the case for $i=0$, note that our conventions simplify the formula to be
\begin{align*}
&(-1)^d{m+d-1\choose d-1}-\rho(-1)^{d}+\rho(-1)^{d-1}{d-1\choose d-2}+\rho\sum_{j=1}^{d-2}{d\choose j}(-1)^j+\sum_{j=1}^{d-1}\left({m+d\choose j}-\rho{d\choose j}\right)(-1)^j\\
&=(-1)^d{m+d-1\choose d-1}-\rho(-1)^{d}+\rho(-1)^{d-1}{d-1\choose d-2}-\rho(-1)^{d-1}{d\choose d-1}+\sum_{j=1}^{d-1}{m+d\choose j}(-1)^j\\
&=(-1)^d{m+d-1\choose d-1}-1+\sum_{j=0}^{d-1}{m+d\choose j}(-1)^j.\\
\end{align*}
We have seen these terms before---as a porism to Proposition \ref{prop:charpoly}, this prior sum simplifies to be $-1$ (see the last lines of the proof to the Proposition). We expected this---the constant term of the KL polynomial is always 1 \cite[Proposition 2.11]{firstklpolys}. Moreover, by the conventions we have taken, note that 
\[-\#SkYT(m+1,0,d+1)+\rho\#\overline{SkYT}(m+1,0,d+1)=-1,\]
as desired.

For the case where $i> 0$, we have done all the hard work in the prior subsection---now what is left is to simply rewrite the equation to be able to utilize the identities. We first combine the parts with the ${d\choose j}$ and $\#SkYT$ to get rewrite it as
\begin{align*}
&(-1)^{d-i}{m+d-1\choose d-i}-\rho(-1)^{d-1}\delta(i,1)+\rho(-1)^{d-1-i}{d\choose d-1-i}\\
&-\rho\sum_{j=1}^{d-2}{d\choose j}\sum_{k=0}^i(-1)^{j-i+k}{j\choose j-i+k}\#\overline{SkYT}(k,d-j-2k+1)\\
&+\sum_{j=1}^{d-1}{m+d\choose j}\sum_{k=0}^i(-1)^{j-i+k}{j\choose j-i+k}\#SkYT(m+1,k,d-j-2k+1)\\
&-\rho{d\choose d-1}\sum_{k=0}^i(-1)^{d-1-i+k}{d-1\choose d-1-i+k}\#SkYT(m+1,k,2-2k)\\
\end{align*}
For this last term, observe that $\#SkYT(m+1,k,2-2k)$ takes on two values: 1 if $k=0$, and 0 otherwise (since $k>1$ implies $2-2k<2$). The same casework allows us to extend the index for the summand for $\#\overline{SkYT}$ to $d-1$, since in either of the cases $\#\overline{SkYT}(k,2-2k)=0$. Note further that if we extend our sums over $j$ so that they start at $j=0$, since ${0\choose 0-i+k}=\delta(i,k)$, the extra terms we get are $\#SkYT(m+1,i,d-2i+1)$ and $-\rho\#\overline{SkYT}(m+1,i,d-2i+1)$. But recall that our goal is to show the above sum equals $-\#SkYT(m+1,i,d-2i+1)+\rho\#\overline{SkYT}(i,d-2i+1)$, so effectively all we have done by allowing $j$ to start at 0 was change the problem to show the following is true:
\begin{align*}
0=&(-1)^{d-i}{m+d-1\choose d-i}-\rho(-1)^{d-1}\delta(i,1)+\rho(-1)^{d-1-i}{d\choose d-1-i}-\rho(-1)^{d-1-i}{d\choose d-1}{d-1\choose d-1-i}\\
&-\rho\sum_{j=0}^{d-1}{d\choose j}\sum_{k=0}^i(-1)^{j-i+k}{j\choose j-i+k}\#\overline{SkYT}(k,d-j-2k+1)\\
&+\sum_{j=0}^{d-1}{m+d\choose j}\sum_{k=0}^i(-1)^{j-i+k}{j\choose j-i+k}\#SkYT(m+1,k,d-j-2k+1).\\
\end{align*}
Note that ${d\choose d-1-i}-{d\choose d-1}{d-1\choose d-1-i}=-i{d\choose i+1}$ and $\# \overline{SkYT}(k,d-j-2k+1)=0$ when $k=0$, and so this equality, and hence the theorem, is true by the identities from Lemmas \ref{lem:SkYT_Identity} and \ref{lem:SkYTbar_Identity}.
\end{proof}


\begin{thebibliography}{99}
\bibitem{deletion}
T. Braden, A. Vysogortes
\textit{Kazhdan-Lusztig Polynomials of Matroids Under Deletion.} arXiv: 1909.09888.

\bibitem{firstklpolys}
B. Elias, N. Proudfoot, M. Wakefield, 
\textit{The Kazhdan-Lusztig polynomial of a matroid}. 
Adv. Math. 299 (2016), 36–-70.

\bibitem{thag}
K. Gedeon
\textit{Kazhdan-Lusztig polynomials of thagomizer matroids}.
Electron. J. Combin. 24 (2017), no. 3, Paper 3.12, 10 pp.
05B35 (05A15 05E05 20C30)


\bibitem{klum} 
A. Gao, L. Lu, M. Xie, A. Yang, P. Zhang.
\textit{The Kazhdan-Lusztig polynomials of uniform matroids}. arXiv: 1806.10852.

\bibitem{klpolysequiv}
K. Gedeon, N. Proudfoot, B. Young,
The equivariant Kazhdan-Lusztig polynomial of a matroid. 
J. Combin. Theory Ser. A 150 (2017), 267–-294.
05B35 (05E05 05E18 20C30)

\bibitem{klpolys}
K. Gedeon, N. Proudfoot, B. Young
\textit{Kazhdan-Lusztig polynomials of matroids: a survey of results and conjectures.} arXiv: 1806.10852.

\bibitem{stirling}
T. Karn, M. Wakefield, 
\textit{Stirling numbers in braid matroid Kazhdan-Lusztig polynomials}. 
Adv. in Appl. Math. 103 (2019), 1-–12.


\bibitem{fanwheelwhirl}
L. Lu, M. Xie, A. Yang
\textit{Kazhdan-Lusztig polynomials of fan matroids, wheel matroids
and whirl matroids}. arXiv: 1802.03711.

\bibitem{kls}
N. Proudfoot, 
\textit{The algebraic geometry of Kazhdan-Lusztig-Stanley polynomials}, 
EMS Surv. Math. Sci. 5 (2018), no. 1, 99-–127.

\bibitem{qniform}
N. Proudfoot
\textit{Equivariant Kazhdan-Lusztig polynomials of $q$-niform matroids.} arXiv: 1808.07855

\bibitem{S}
R. Stanley, \textit{Polygon dissections and standard Young tableaux},  
Journal of Combinatorial Theory, series A \textbf{76} (1996), 175--177. 

\bibitem{pfia}
M. Wakefield
\textit{Partial flag incidence algebras.} arXiv: 1605.01685

\bibitem{equithag}
M. Xie, P. Zhang
 \textit{Equivariant Kazhdan-Lusztig polynomials of thagomizer matroids}.
Proc. Amer. Math. Soc. 147 (2019), no. 11, 4687-–4695.
05B35 (05E05 20C30)

\end{thebibliography}
\end{document}